\documentclass[11pt,a4paper]{article}
\usepackage{amssymb,amsmath}

\newtheorem{theorem}{Theorem}[section]
\newtheorem{corollary}[theorem]{Corollary}
\newtheorem{definition}[theorem]{Definition}
\newtheorem{lemma}[theorem]{Lemma}
\newtheorem{proposition}[theorem]{Proposition}
\newtheorem{remark}[theorem]{Remark}

\newtheorem{example}[theorem]{Example}

\def\mn{{\mathbb{N}}}

\newenvironment{proof}{\begin{trivlist}\item[]{\it
Proof.}}{\hfill$\square$\end{trivlist}}

\begin{document}
%
\title{On syzygies for rings of invariants of abelian groups}
%
%
\author{
M. Domokos\thanks{This research was partially supported by National Research, Development and Innovation Office,  NKFIH K 119934.}  
\\
{\small MTA R\'enyi Institute, Budapest, Hungary}
\\  {\small E-mail: domokos.matyas@renyi.mta.hu } 
}
\date{}


\maketitle              

\begin{abstract} 
It is well known that results on zero-sum sequences over a finitely generated abelian group 
can be translated to statements on generators of rings of invariants of the dual group. Here  the direction of the transfer of information between zero-sum theory and invariant theory is reversed. First it is shown how a presentation by generators and relations of the ring of invariants of an abelian group acting linearly on a finite dimensional vector space can be obtained from a presentation of the ring of invariants for the corresponding multiplicity free representation. This  combined 
with a known degree bound for syzygies of rings of invariants,  yields bounds on the 
presentation of a block monoid associated to a finite sequence of elements in an abelian group. 
The results have an equivalent formulation in terms of binomial ideals, but 
here  the language of monoid congruences and the notion of catenary degree is used. 

\end{abstract}

\section{Introduction}\label{sec:intro}

Let $G$ be an  abelian group (written multiplicatively) and $G_0\subseteq G$ a finite subset. 
Consider the additive monoid $\mathbb{N}_0^{G_0}$ (maps from $G_0$ into the set of non-negative integers, with pointwise addition). It contains the submonoid 
\begin{equation}\label{eq:def_block_monoid}
\mathcal{B}(G_0):=\{\alpha\in \mathbb{N}_0^{G_0}\mid \prod_{g \in G_0}
g^{\alpha(g)}=1\in G\}
\end{equation} 
called the {\it block monoid of $G_0$} or the {\it monoid of product-one sequences over $G_0$}, see \cite[Definition 3.4.1]{geroldinger_halter-koch}. It causes no loss of generality in the construction of $\mathcal{B}(G_0)$ if we assume that the group $G$ is generated by $G_0$. 
We note that in most of the related literature the group is written additively, and 
therefore the terminology of \emph{zero-sum sequences} is used. 
$\mathcal{B}(G_0)$ is a reduced affine monoid 
(see for example \cite[Theorem 3.4.2.1]{geroldinger_halter-koch}), write $\mathcal{A}(\mathcal{B}(G_0))$ for the finite set of its atoms. 

Our focus is on a similar but more general construction. 
Fix an $m$-tuple $\underline{g}=(g_1,\dots,g_m)\in G^m$ of elements of the abelian group $G$, and set 
\begin{equation}\label{eq:B(g)} 
\mathcal{B}(\underline{g})=\{\alpha\in\mathbb{N}_0^m\mid 
\prod_{i=1}^mg_i^{\alpha_i}=1\in G\}.\end{equation}  
This is a finitely generated submonoid of the additive monoid $\mathbb{N}_0^m$. 
Write $\mathrm{supp}(\underline{g})$ for the subset $\{g_1,\dots,g_m\}$ of $G$. 
In the special case when $g_1,\dots,g_m$ are distinct, the monoid $\mathcal{B}(\underline{g})$ can be identified with $\mathcal{B}(\mathrm{supp}(\underline{g}))$. 
In general (i.e. when $g_1,\dots,g_m$ are not all distinct) the monoid 
$\mathcal{B}(\underline{g})$ is different from $\mathcal{B}(\mathrm{supp}(\underline{g}))$. 
So the construction \eqref{eq:B(g)} is indeed a generalization of  \eqref{eq:def_block_monoid} (however, see Remark~\ref{remark:G_0-g} in Section~\ref{sec:repetition}). 

Interest in the monoids $\mathcal{B}(G_0)$ and $\mathcal{B}(\underline{g})$ and their semigroup rings comes from several mathematical topics: 
factorization theory in monoids, multiplicative ideal theory, zero-sum theory, invariant theory, 
toric varieties, binomial or toric ideals. 
For example, it has been long known that results on the atoms in monoids of the form $\mathcal{B}(\underline{g})$ can be reformulated in terms of generators or degree bounds for  
rings of invariants of abelian groups (see e.g. \cite{cziszter_domokos_geroldinger} for some details and references). 

In this paper we shall study presentations of reduced affine monoids, with a paticular attention on the monoids $\mathcal{B}(\underline{g})$. 
By a {\it monoid} we mean a commutative cancellative semigroup with an identity element. 
A monoid is {\it affine} if it is a finitely generated submonoid of a finitely generated free abelian froup. 
We say that a monoid is {\it reduced} if the identity element is its only unit (invertible element). Recall that by Grillet's theorem, reduced affine monoids are exactly the monoids that are isomorphic to a finitely generated submonoid of the additive monoid $\mathbb{N}_0^k$ for some $k$ (see for example \cite[Proposition 2.16]{bruns_gubeladze}). 

Let $S$ be a reduced affine monoid (written multiplicatively) and $\mathcal{A}(S)$ the set of atoms in $S$. 
Then $\mathcal{A}(S)$ is finite, and it is a minimal generating set of $S$, see for example Proposition 1.1.7 in \cite{geroldinger_halter-koch}. Denote by 
$\mathbb{N}_0^{\mathcal{A}(S)}$ the additive monoid of functions from $\mathcal{A}(S)$ into the additive monoid $\mathbb{N}_0$ of non-negative integers; this is a free monoid generated by $|\mathcal{A}(S)|$ elements. 
Take  commuting indeterminates $\{x_a \mid a\in \mathcal{A}(S)\}$,  and let $M$ denote the free monoid  generated by them (written multiplicatively): 
\[M=\{x^{\alpha}= \prod_{a\in\mathcal{A}(S)}x_a^{\alpha(a)} \mid \alpha
\in \mathbb{N}_0^{\mathcal{A}(S)}\}\] 
with multiplication $x^{\alpha}x^{\gamma}=x^{\alpha+\gamma}$.  
Consider  the  unique semigroup homomorphism 
\[\pi: M\to S \mbox{ given by }
x_a\mapsto a, \quad a\in \mathcal{A}(S)\] 
(called {\it factorization homomorphism} in \cite{blanco_etal}).  
Denote by $\sim_S$ the congruence on $M$  defined by 
\[x^{\alpha}\sim_S x^{\gamma} \iff \pi(x^{\alpha})=\pi(x^{\gamma})\in S,\] 
and call it the {\it defining congruence of $S$} (it is called {\it the monoid of relations of $S$} in \cite{blanco_etal}).   
The semigroup homomorphism $\pi$ factors through a monoid isomorphism 
\[M/\sim_S\ \stackrel{\cong}\longrightarrow S.\] 
Formally  the  congruence $\sim_S$ will be viewed as a subset of 
$M\times M$. 
The congruence $\sim_S$ is finitely generated by \cite{redei}.  
By a {\it presentation} of $S$ we mean a finite subset of $M\times M$ generating the conguence 
$\sim_S$ (see for example \cite[Section I.4]{gilmer} for basic notions related to semigroup congruences). 

Now let us summarize the content of the present paper. 
Our Theorem~\ref{thm:tildeS_presentation} tells in particular how a presentation of 
$\mathcal{B}(\underline{g})$ can be derived from a presentation of $\mathcal{B}(\mathrm{supp}(\underline{g}))$. It turns out that in most cases the catenary degree 
(cf. Definition~\ref{def:cat_degree})  of 
$\mathcal{B}(\underline{g})$ coincides with the catenary degree of 
$\mathcal{B}(\mathrm{supp}(\underline{g}))$ (see Corollary~\ref{cor:B(chi-tilde)}). 
Combining Theorem~\ref{thm:tildeS_presentation}  with a degree bound of Derksen \cite{derksen} for the defining relations of the ring of invariants of a linearly reductive group, 
we derive in Theorem~\ref{thm:c_w(B(g))} degree bounds for a presentation of 
$\mathcal{B}(\underline{g})$. In order to formulate this result we introduce the notion of 
graded catenary degree of a graded monoid (see Definition~\ref{def:graded_cat_degree}), which is a refinement of (and an upper bound for) the ordinary catenary degree. These results on presentations   of monoids have an equivalent formulation in terms of generators of (binomial) ideals of relations of semigroup rings of monoids, this is  pointed out in Section~\ref{sec:ringtheoretic-catenarydegree}. 
A Gr\"obner basis version of Theorem~\ref{thm:tildeS_presentation} is given in  
Theorem~\ref{thm:groebner}. Since the semigroup rings of the form  
$\mathbb{C}[\mathcal{B}(G)]$ are exactly the rings of invariants of abelian groups, 
the results have relevance for toric varieties;  
this is expanded a bit in Section~\ref{sec:toric}, and some examples of toric quiver varieties are reviewed.  
We point out finally that Theorem~\ref{thm:groebner} provides a source of Koszul  algebras.


\section{Ring theoretic characterization of the catenary degree}\label{sec:ringtheoretic-catenarydegree}

From now on $S$ will be a reduced affine monoid. 
The {\it catenary degree}  $c(S)$ of the monoid $S$ is a basic arithmetic invariant studied in factorization theory, let us recall its definition.  For $\alpha\in \mathbb{N}_0^{\mathcal{A}(S)}$ we set 
\[|\alpha|=\sum_{a\in\mathcal{A}(S)}|\alpha(a)|\]  
and 
\[a^{\alpha}:=
\prod_{a\in\mathcal{A}(S)}a^{\alpha(a)} \in S.\]  
For $\alpha,\gamma\in  \mathbb{N}_0^{\mathcal{A}(S)}$ we write $\mathrm{gcd}(\alpha,\gamma)$ for the greatest common divisor of $\alpha,\gamma$ in the additive monoid $\mathbb{N}_0^{\mathcal{A}(S)}$ 
(i.e. $\mathrm{gcd}(\alpha,\gamma)(a)=\min\{\alpha(a),\gamma(a)\}$ for each $a\in \mathcal{A}(S)$). 

\begin{definition}\label{def:cat_degree} {\rm (see 
\cite[Definition 1.6.1]{geroldinger_halter-koch})}
For $\alpha,\gamma\in \mathbb{N}_0^{\mathcal{A}(S)}$ set 
\[d(\alpha,\gamma):=\max\{|\alpha-\mathrm{gcd}(\alpha,\gamma)|,
|\gamma-\mathrm{gcd}(\alpha,\gamma)|\}.\]
Given the monoid $S$ as above, we say that $\alpha$ and $\gamma$ can be connected by a $d$-chain if there exists a sequence $\alpha^{(0)}=\alpha,\alpha^{(1)},\dots,\alpha^{(k)}=\gamma\in
\mathbb{N}_0^{\mathcal{A}(S)}$ such that $a^{\alpha^{(j)}}= a^{\alpha^{(j+1)}}$ and $d(\alpha^{(j)},\alpha^{(j+1)})\le d$ for $j=0,1,\dots,k-1$. 
The \emph{catenary degree} $c(S)$ is the minimal non-negative integer $d$ such that if $a^{\alpha}=a^{\gamma}$, 
then $\alpha$ and $\gamma$ can be connected by a $d$-chain. 
\end{definition}

\begin{remark} Note that $c(S)=0$ if and only if $S$ is a free (or factorial) monoid 
(i.e. $S$ is isomorphic to the additive monoid $\mathbb{N}_0^m$ for some $m$), and 
$c(S)$ is never equal to $1$. 
\end{remark}

Characterizations of the catenary degree and its variants are given in \cite{Ph10a}, \cite{Ph15a}. In particular, \cite[Proposition 16]{Ph10a} characterizes the catenary degree in terms of the \emph{monoid} of relations. 
Now extending an observation from \cite{chapman_etal} we formulate a characterization of the catenary degree in terms of semigroup congruences. 
Denote by $c'(S)$  the minimal non-negative integer $d$ such that there exists a 
generating set $\Lambda\subset M\times M$ of the semigroup congruence $\sim_S$,  satisfying that 
for all $(x^{\alpha},x^{\gamma})\in \Lambda$ we have $|\alpha|\le d$ and $|\gamma|\le d$.  
An explicit description of the semigroup congruence generated by a subset of $M\times M$ can be found for example in \cite[page 176]{herzog}. 

\begin{proposition}\label{prop:catenary} 
We have $c(S)=c'(S)$.  
\end{proposition} 

\begin{proof} 
Set 
\[\Lambda:=\{(x^{\alpha},x^{\gamma})\mid x^{\alpha}\sim_S x^{\gamma} 
\mbox{ and }|\alpha|\le c(S),\ |\gamma|\le c(S)\}.\] 
We claim that $\Lambda$ generates the semigroup congruence $\sim_S$. 
Indeed, take a pair $(\alpha,\gamma)\in \mathbb{N}_0^{\mathcal{A}(S)}\times  \mathbb{N}_0^{\mathcal{A}(S)}$ such that $x^{\alpha}\sim_S x^{\gamma}$. Then 
$\alpha$ and $\gamma$ can be connected by a $c(S)$-chain $\alpha^{(0)}=\alpha,\alpha^{(1)},\dots,\alpha^{(k)}=\gamma$. For $i=0,1,\dots,k-1$ the pair 
$(\alpha^{(i)},\alpha^{(i+1)})$ is of the form $(\beta^{(i)}+\delta^{(i)},\rho^{(i)}+\delta^{(i)})$, 
where $\delta^{(i)},\beta^{(i)},\rho^{(i)}\in \mathbb{N}_0^{\mathcal{A}(S)}$,  
$|\beta^{(i)}|\le c(S)$, $|\rho^{(i)}|\le c(S)$. Moreover, since $S$ is cancellative, $a^{\alpha^{(i)}}=a^{\alpha^{(i+1)}}$, implies  $a^{\beta^{(i)}}=a^{\rho^{(i)}}$,  so $x^{\beta^{(i)}}\sim_S x^{\rho^{(i)}}$. 
Thus $(x^{\beta^{(i)}}, x^{\rho^{(i)}})\in \Lambda$. It follows that 
$(x^{\alpha^{(i)}},x^{\alpha^{(i+1)}})=(x^{\beta^{(i)}}x^{\delta^{(i)}},x^{\rho^{(i)}}x^{\delta^{(i)}})$ belongs to the congruence generated by $\Lambda$ for $i=0,\dots,k-1$, implying in turn that $(x^{\alpha},x^{\gamma})=(x^{\alpha^{(0)}},x^{\alpha^{(k)}})$ belongs to the congruence generated by $\Lambda$. 
This proves the inequality $c'(S)\le c(S)$. 

The reverse inequality $c(S)\le c'(S)$ is 
pointed out in \cite[Proposition 3.1]{chapman_etal}.  
\end{proof}

Fix a commutative ring $R$ (having an identity element) and consider the semigroup rings $R[M]$ and $R[S]$. Note that $R[M]=R[x_a\mid a\in\mathcal{A}(S)]$ is the polynomial ring over $R$ with indeterminates $\{x_a\mid a\in \mathcal{A}(S)\}$. The  monoid homomorphism $\pi:M\to S$ extends uniquely to an $R$-algebra 
homomorphism 
\begin{equation}\label{eq:pi_R}
\pi_R:R[x_a \mid a\in\mathcal{A}(S)]\to R[S], \quad  \pi(x^{\alpha})=a^{\alpha}.
\end{equation} 
The statement below is well known, see \cite[Proposition 1.5]{herzog} or 
\cite[Chapter II.7.]{gilmer}: 

\begin{proposition}\label{prop:catenary_char} 
The following conditions are equivalent for a set of pairs   
$B\subset \mathbb{N}_0^{\mathcal{A}(S)}\times \mathbb{N}_0^{\mathcal{A}(S)}$: 
\begin{itemize} 
\item[(1)] The semigroup congruence $\sim_S$ is   generated by $\{(x^{\alpha},x^{\gamma}) \mid (\alpha,\gamma)\in B\}$. 
\item[(2)] The ideal $\ker(\pi_R)$ is generated by the binomials $\{x^{\alpha}-x^{\gamma}\mid (\alpha,\gamma)\in B\}$. 
\end{itemize} 
\end{proposition}

\begin{remark}\label{remark:ring_does_not_matter}
Condition (1) of Proposition~\ref{prop:catenary_char} does not depend on the ring $R$, 
therefore a set of binomials generates the ideal $\ker(\pi_R)$ for some ring $R$ if and only if it generates $\ker(\pi_R)$ for any ring $R$. 
\end{remark} 

Proposition~\ref{prop:catenary} and Proposition~\ref{prop:catenary_char} imply  the following ring theoretic characterization of $c(S)$: 

\begin{corollary}\label{cor:catenary_char}
The catenary degree $c(S)$ is the minimal positive integer $d$ such that the kernel 
of $\pi_R:R[x_a\mid a\in\mathcal{A}(S)]\to R[S]$ is generated (as an ideal) by binomials of degree at most $d$ for some (hence any) commutative ring $R$ (where  
$R[x_a\mid a\in\mathcal{A}(S)]$ is graded in the standard way, namely the generators 
$x_a$ have degree $1$ and the non-zero scalars in $R$ have degree $0$). 
\end{corollary} 


\section{Graded monoids}\label{sec:graded_monoid}

Let $S$ be a  {\it graded} monoid; that is, $S$ is partitioned into the disjoint union of subsets 
$S_d$, $d\in\mathbb{N}_0$, such that $S_d\cdot S_e\subseteq S_{d+e}$. For $s\in S$ write $|s|=d$ if $s\in S_d$. The identity element of $S$ necessarily belongs to $S_0$. We call a graded monoid {\it connected graded} if $S_0$ consists only of the identity element.  
 It seems natural to modify Definition~\ref{def:cat_degree} for graded monoids as follows: 
\begin{definition}\label{def:graded_cat_degree}
Given a connected graded reduced affine monoid $S$,
for $\alpha,\gamma\in \mathbb{N}_0^{\mathcal{A}(S)}$ set 
\[|\alpha|_{\mathrm{gr}}=\sum_{a\in\mathcal{A}(S)}\alpha(a)|a|\] 
and 
\begin{align*}
d_{\mathrm{gr}}(\alpha,\gamma):=
\max\{|\alpha-\mathrm{gcd}(\alpha,\gamma)|_{\mathrm{gr}},
|\gamma-\mathrm{gcd}(\alpha,\gamma)|_{\mathrm{gr}}\}.
\end{align*}
We say that \emph{$\alpha$ and $\gamma$ can be connected by a chain of weight at most $d$} if there exists a sequence $\alpha^{(0)}=\alpha,\alpha^{(1)},\dots,\alpha^{(k)}=\gamma\in
\mathbb{N}_0^{\mathcal{A}(S)}$ such that $a^{\alpha^{(j)}}= a^{\alpha^{(j+1)}}$ and $d_{\mathrm{gr}}(\alpha^{(j)},\alpha^{(j+1})\le d$ for $j=0,1,\dots,k-1$. 
The \emph{graded catenary degree} $c_{\mathrm{gr}}(S)$ is the minimal $d$ such that if $a^{\alpha}=a^{\gamma}$, 
then $\alpha$ and $\gamma$ can be connected by a chain of weight at most $d$. 
\end{definition}

As a straightforward modification of Proposition~\ref{prop:catenary} we get the following. 

\begin{proposition} \label{prop:graded-catenary} 
The graded catenary degree $c_{\mathrm{gr}}(S)$ is the minimal non-negative integer $d$ such that 
there exists a 
generating set $\Lambda\subset M\times M$ of the semigroup congruence $\sim_S$ satisfying  that 
for all $(x^{\alpha},x^{\gamma})\in \Lambda$ we have 
$|\alpha|_{\mathrm{gr}}\le d$ and 
$|\gamma|_{\mathrm{gr}}\le d$. 
\end{proposition}

The grading of $S$ induces a grading on the semigroup ring 
$R[S]=\bigoplus_{d=0}^{\infty}R[S]_d$, where $R[S]_d$ is the $R$-submodule generated by $S_d$. Lift the grading to $M$ and $R[M]$ by setting the degree of $x_a$ to be equal to the degree $|a|$ of $a\in\mathcal{A}(S)$. Then the map 
$\pi_R:R[x_a\mid a\in\mathcal{A}(S)]\to R[S]$ is a homomorphism of graded algebras, and so $\ker(\pi_R)$ is a homogeneous ideal. 
Our assumptions on the grading imply that all indeterminates $x_a$ have positive degree. 
Moreover, $x^{\alpha}\sim_S x^{\gamma}$ implies that $a^{\alpha}$ and $a^{\gamma}$ belong to the same homogeneous component of $S$, and therefore the binomials in $\ker(\pi_R)$ are homogeneous. For an ideal $I$ in 
$R[x_a\mid a\in\mathcal{A}(S)]$ 
denote  by $\mu(I)$ the minimal non-negative integer $d$ such that $I$ is generated by elements of degree at most $d$; this number is finite for any binomial ideal $I$. 

\begin{corollary} \label{cor:mu=c}
For any connected graded reduced affine monoid we have the equality 
\[c_{\mathrm{gr}}(S)=\mu(\ker(\pi_R)),\] 
where $R[x_a\mid a\in \mathcal{A}(S)]$ is endowed with the grading that makes $\pi_R$ a homomorphism of graded algebras. 
\end{corollary} 

\begin{proof}  
Recall that any homogeneous generating system of a homogeneous ideal $I$ contains a minimal (with respect to inclusion) homogeneous generating system, and $\mu(I)$ is the maximal degree of an element in any minimal homogeneous generating system 
(this follows from the graded Nakayama lemma). 
The ideal $\ker(\pi_R)$ is generated by homogeneous binomials. 
Take a minimal set of binomials generating $\ker(\pi_R)$. 
Then the maximal degree of an element in this set of binomials equals $\mu(\ker(\pi_R))$ on one hand, and it equals $c_{\mathrm{gr}}(S)$ 
by Proposition~\ref{prop:catenary_char} and 
Proposition~\ref{prop:graded-catenary}. 
\end{proof}


\section{Repetition of elements} \label{sec:repetition} 

A surjective monoid homomorphism $\theta:T\to B$ between reduced affine monoids $T$ and $B$ is called a \emph{transfer homomorphism} if for any $t\in T$, $b,c\in B$ with $\theta(t)=bc$, 
there exist elements $u,v\in T$ such that $t=uv$ and $\theta(u)=b$, $\theta(v)=c$ 
(see \cite[Definition 3.2.1]{geroldinger_halter-koch}). 

Now let $\underline{g}$ be an $m$-tuple of elements in an abelian group $G$ and 
$\mathcal{B}(\underline{g})$, $\mathcal{B}(\mathrm{supp}(\underline{g}))$ the monoids introduced in Section~\ref{sec:intro}. 
The map 
\begin{equation}\label{eq:B(g)_transfer}
\mathbb{N}_0^m\to \mathbb{N}_0^{\mathrm{supp}(\underline{g})},\quad \alpha\mapsto (g\mapsto 
\sum_{g_i=g}\alpha_i)
\end{equation} 
restricts to a {\it transfer homomorphism}  
$\mathcal{B}(\underline{g})\to \mathcal{B}(\mathrm{supp}(\underline{g}))$. 
In factorization theory transfer homomorphisms are used to 
reduce the computation of arithmetic invariants of monoids to the corresponding invariants of other monoids (frequently in block monoids). 
In particular, it is known that the catenary degrees of monoids connected by a transfer homomorphism 
are linked as follows: 

\begin{lemma}\label{lemma:catenary_transfer} \cite[Theorem 3.2.5.5]{geroldinger_halter-koch} 
Let $\theta:T\to B$ be a transfer homomorphism, where $T$ and $B$ are reduced affine monoids. Then we have the inequalities 
\[c(B)\le c(T)\le \max\{c(B),c(T,\theta)\}\] 
(see  \cite[page 171]{geroldinger_halter-koch} for the definition of $c(T,\theta)$). 
\end{lemma} 

The aim of this section is to refine Lemma~\ref{lemma:catenary_transfer} for the transfer homomorphism 
$\mathcal{B}(\underline{g})\to \mathcal{B}(\mathrm{supp}(\underline{G}))$ given by 
\eqref{eq:B(g)_transfer}. More precisely, it will be shown how one can get a generating system of the defining congruence of $\mathcal{B}(\underline{g})$ from a given generating system of the defining congruence of $\mathcal{B}(\mathrm{supp}(\underline{g}))$. 
This will be done in a more general setup. 

Assume  that $S$ is a (not necessarily connected) graded, reduced, affine monoid and denote by 
$\tilde S$ the monoid 
\begin{equation}\label{eq:tildeS} 
\tilde S=\{s[i]\mid s\in S, 0\le i\le |s|\} \mbox{ with multiplication }
s[i]\cdot t[j]=(s\cdot t)[i+j]. \end{equation}  
So $\tilde S$ is a submonoid of the direct product of $S$ and the additive monoid $\mathbb{N}_0$. 
Obviously the map $\tilde S\to S$, $s[i]\mapsto s$ is a transfer homomorphism, and   
\[\mathcal{A}(\tilde S)=\{a[i]\mid a\in\mathcal{A}(S),\ 0\le i\le |a|\}.\] 
This transfer homomorphism $\tilde S\to S$ induces a monoid homomorphism 
\[\kappa:\mathbb{N}_0^{\mathcal{A}(\tilde S)}\to \mathbb{N}_0^{\mathcal{A}(S)},\quad 
\lambda\mapsto \left(a\mapsto \sum_{i=0}^{|a|}\lambda(a[i])\right).\] 
Set  
\[\delta:\mathbb{N}_0^{\mathcal{A}(\tilde S)}\to \mathbb{N}_0,\quad 
\lambda\mapsto \sum_{a[i]\in\mathcal{A}(\tilde S)}i\lambda(a[i]).\] 
Note that for $\lambda,\mu\in \mathbb{N}_0^{\mathcal{A}(\tilde S)}$ we have 
\begin{equation}\label{eq:kappa_delta} 
x^{\lambda}\sim_{\tilde S}x^{\mu} \iff 
x^{\kappa(\lambda)}\sim_S x^{\kappa(\mu)}\mbox{ and }\delta(\lambda)=\delta(\mu). 
\end{equation} 
Also we have 
$\delta(\lambda)\le |a^{\kappa(\lambda)}|= |a^{\kappa(\mu)}|$. 

\begin{lemma}\label{lemma:quadratic} 
Assume that $\alpha,\beta\in\mn_0^{\mathcal{A}(\tilde S)}$ satisfy 
\[\kappa(\alpha)=\kappa(\beta) \quad \mbox{ and }
\delta(\alpha)=\delta(\beta).\]
Then $x^{\alpha}\sim a^{\beta}$ with respect to the semigroup congruence $\sim$ on the free monoid 
$\tilde M=\{x^{\alpha}\mid \alpha\in \mathbb{N}_0^{\mathcal{A}(\tilde S)}\}$ 
generated by  
\[\{(x_{a[k]}x_{b[l]}, x_{a[k+1]}x_{b[l-1]})\mid  
 a,b\in\mathcal{A}(S),\ 0\le k\le |a|-1,\ 
1\le l\le |b|\}.\] 
\end{lemma}
\begin{proof}  Case I: $x^{\alpha}$ and $x^{\beta}$ involve a common variable $x_{a[i]}$.  Then 
$x^{\alpha}=x_{a[i]}x^{\alpha'}$, $x^{\beta}=x_{a[i]}x^{\beta'}$, and 
the assumptions on the pair $(\alpha,\beta)$ in the lemma hold for the pair 
$(\alpha',\beta')$.  By induction on $\displaystyle \sum_{a[i]\in\mathcal{A}(\tilde S)}\alpha(a[i])$ we may conclude that $x^{\alpha'}\sim x^{\beta'}$, implying in turn that $x^{\alpha}\sim x^{\beta}$. 

Case II: $x^{\alpha}$ and $x^{\beta}$ are not divisible by a common variable. Take a variable $x_{a[i]}$ dividing 
$x^{\alpha}$. Then $\kappa(\alpha)=\kappa(\beta)$ implies that  
$x^{\beta}$ is divisible by $x_{a[k]}$ for some $k\neq i$. By symmetry it is sufficient to deal with the case $i>k$. By the assumptions on $\alpha,\beta$ there must exist an atom  $b\in\mathcal{A}(S)$ and integers $j<l$ such that $x_{a[i]}x_{b[j]}$ divides $x^{\alpha}$ and $x_{a[k]}x_{b[l]}$ divides $T^{\beta}$. We have 
$x^{\alpha}=x_{a[i]}x_{b[j]}x^{\gamma}\sim x_{a[i-1]}x_{b[j+1]}x^{\gamma}=x^{\alpha'}$. 
The conditions of the lemma on the pair $(\alpha,\beta)$ hold also for the pair $(\alpha',\beta)$.  If $i-1=k$, then $x^{\alpha'}$ and $x^{\beta}$  are divisible by a common variable, and we are back in Case I. 
Otherwise  similarly to the above process we have 
$x^{\alpha'}\sim x^{\alpha''}$ where $x^{\alpha''}$ is divisible by $x_{a[i-2]}$ and 
$\kappa(\alpha'')=\kappa(\beta)$ and 
$\delta(\alpha'')=\delta(\beta)$. 
After finitely many such steps we get back to Case I.   
\end{proof} 

\begin{remark} 
Lemma~\ref{lemma:quadratic} says that for the transfer homomorphism 
$\theta:\tilde S\to S$, $s[i]\mapsto s$ we have $c(\tilde S,\theta)\le 2$ in 
Lemma~\ref{lemma:catenary_transfer}. 
\end{remark} 

\begin{theorem}\label{thm:tildeS_presentation} 
Suppose that the congruence $\sim_S$ is generated by 
$\{(x^{\lambda}, x^{\mu})\mid (\lambda,\mu)\in\Lambda\}$ for some 
$\Lambda\subset \mathbb{N}_0^{\mathcal{A}(S)}\times \mathbb{N}_0^{\mathcal{A}(S)}$. 
For each $\lambda\in \mathbb{N}_0^{\mathcal{A}(\tilde S)}$ such that $(\lambda,\mu)\in \Lambda$ or $(\mu,\lambda)\in \Lambda$ for some $\mu$, and for each 
$0\le i\le \delta(\lambda)$ choose 
$\lambda[i]\in \mathbb{N}_0^{\mathcal{A}(\tilde S)}$ 
such that $\kappa(\lambda[i])=\lambda$ and $\delta(\lambda[i])=i$
(this is clearly possible). 
Then the congruence $\sim_{\tilde S}$ is generated by 
\begin{align*}\{(x^{\lambda[i]}, x^{\mu[i]}),\ 
(x_{a[k]}x_{b[l]}, x_{a[k+1]}x_{b[l-1]}) 
\mid 
(\lambda,\mu)\in \Lambda, \ 0\le i\le |a^{\lambda}|, \\ 
a,b\in\mathcal{A}(S),\ 0\le k\le |a|-1,\ 
1\le l\le |b|\}.
\end{align*} 
\end{theorem}

\begin{proof} The pairs given in the statement do belong to the congruence 
$\sim_{\tilde S}$. 
Denote by $\sim$ the semigroup congruence on 
$\tilde M=\{x^{\alpha}\mid \alpha\in \mathbb{N}_0^{\mathcal{A}(\tilde S)}\}$ generated by them. It is sufficient to show that if $x^{\alpha}\sim_{\tilde S} x^{\beta}$ 
for some $\alpha,\beta\in \mathbb{N}_0^{\mathcal{A}(\tilde S)}$, then 
$x^{\alpha}\sim x^{\beta}$. By \eqref{eq:kappa_delta}  we have 
$x^{\kappa(\alpha)}\sim_S x^{\kappa(\beta)}$  and $\delta(\alpha)=\delta(\beta)$. 
Therefore there exists a sequence 
$(\lambda_j,\mu_j) \in  \mathbb{N}_0^{\mathcal{A}(S)} \times \mathbb{N}_0^{\mathcal{A}(S)}$ and $\gamma_j\in\mn_0^{\mathcal{A}(S)}$  
$(j=1,\dots,s)$ such that $(\lambda_j,\mu_j)\in \Lambda$ or $(\mu_j,\lambda_j)\in\Lambda$ 
(implying in particular that $|a^{\lambda_j}|=|a^{\mu_j}|$), 
$\lambda_1+\gamma_1=\kappa(\alpha)$,  $\mu_s+\gamma_s=\kappa(\beta)$, and $\mu_j+\gamma_j=\lambda_{j+1}+\gamma_{j+1}$ for 
$j=1,\dots,s-1$. 
Set $d:=\delta(\alpha)=\delta(\beta)$. 
For each $j=1,\dots,s$ choose a non-negative integer $k_j$ with 
\[d-|a^{\gamma_j}|
\le k_j\le 
|a^{\lambda_j}|. \]
This is possible, because 
\[d\le 
|a^{\kappa(\alpha)}|=|a^{\lambda_j+\gamma_j}|
=|a^{\lambda_j}|+|a^{\gamma_j}|.\]
Taking into account Lemma~\ref{lemma:quadratic} we get 
\begin{align*}x^{\alpha}\sim  x^{\kappa(\alpha)[d]}=x^{(\lambda_1+\gamma_1)[d]}\sim x^{\lambda_1[k_1]}x^{\gamma_1[d-k_1]}\sim x^{\mu_1[k_1]}x^{\gamma_1[d-k_1]}\sim 
\\  x^{(\mu_1+\gamma_1)[d]}= x^{(\lambda_2+\gamma_2)[d]}\sim 
x^{\lambda_2[k_2]}x^{\gamma_2[d-k_2]}\sim x^{\mu_2[k_2]}x^{\gamma_2[d-k_2]}\sim \\ 
x^{(\mu_2+\gamma_2)[d]} = x^{(\lambda_3+\gamma_3)[d]}\sim \cdots 
\sim x^{(\mu_s+\gamma_s)[d]}= x^{\kappa(\beta)[d]}\sim x^{\beta}.
 \end{align*} 
\end{proof} 

\begin{remark} \label{remark:hibi}
Statements related to Theorem~\ref{thm:tildeS_presentation} are proved in \cite{ohsugi_hibi}, studying toric ideals associated to nested configurations (see also 
\cite{shibuta} for some generalization).  The construction of $\tilde S$ from $S$ (see 
\eqref{eq:tildeS}) can be seen as a special case of the construction of nested configurations. 
In particular, when $S$ is a submonoid of $\mathbb{N}_0^d$ generated by finitely many elements $\alpha^{(1)},\dots,\alpha^{(m)}$ such that there exists a $v\in\mathbb{R}^d$ with $\sum_{j=1}^d\alpha^{(i)}_jv_j=1$ for all 
$i=1,\dots,m$ (this implies that $S$ can be graded in such a way that each generator has degree $1$), the results of \cite{ohsugi_hibi} apply to the binomial ideal associated 
to the monoid $\tilde S$ and yield a system of generators similar to the one given by 
Theorem~\ref{thm:tildeS_presentation}. 
\end{remark} 

The monoid $\mathcal{B}(\underline{g})$ can be obtained from the monoid 
$\mathcal{B}(\mathrm{supp}(\underline{g}))$ by a repeated application of the construction \eqref{eq:tildeS}, and therefore 
Theorem~\ref{thm:tildeS_presentation} can be applied to relate the catenary degree of 
$\mathcal{B}(\underline{g})$ to the catenary degree of $\mathcal{B}(\mathrm{supp}(\underline{g}))$.  
Indeed, start with an $m$-tuple $\underline{g}=(g_1,\dots,g_m)\in G^m$ of 
not necessarily distinct elements in $G$, and denote by 
$\tilde{\underline{g}}$ the $m+1$-tuple 
$(g_1,\dots,g_m,g_m)$ obtained from $\underline{g}$ by repeating the $m^{\mathrm{th}}$ component. Consider the grading on the monoid $\mathcal{B}(\underline{g})$ given by 
\[\mathcal{B}(\underline{g})_d=\{\alpha\in \mathcal{B}(\underline{g})\subseteq \mathbb{N}_0^m\mid 
\alpha_m=d\}. \]

\begin{proposition}\label{prop:B(tilde-g)} 
We have $\mathcal{B}(\tilde{\underline{g}})\cong \tilde S$, where 
$S=\mathcal{B}(\underline{g})$ is endowed with the above grading and $\tilde S$ is 
defined as in \eqref{eq:tildeS}.  
\end{proposition}

\begin{proof} A general element of $\tilde S$ is of the form $\alpha[i]$ where 
$\alpha=(\alpha_1,\dots,\alpha_m)\in \mathcal{B}(\underline{g})$ and $0\le i\le \alpha_m$. 
The map $\tilde S\to \mathcal{B}(\tilde{\underline{g}})$ sending 
$\alpha[i]\in \tilde S$ to  $(\alpha_1,\dots,\alpha_{m-1},\alpha_m-i,i)$ is an isomorphism 
between the monoids $\tilde S$ and $\mathcal{B}(\tilde{\underline{g}})$. 
\end{proof}

Proposition~\ref{prop:catenary}, Theorem~\ref{thm:tildeS_presentation}, and 
Proposition~\ref{prop:B(tilde-g)} imply the following: 

\begin{corollary}\label{cor:B(chi-tilde)} 
We have $c(\mathcal{B}(\underline{g}))=c(\mathcal{B}(\mathrm{supp}(\underline{g}))$, unless 
$\mathcal{B}(\mathrm{supp}(\underline{g}))$ is a free monoid and the components $g_1,\dots,g_m$ of $\underline{g}$ are not all distinct. 
In the latter case we have $c(\mathcal{B}(\mathrm{supp}(\underline{g})))=0$ whereas 
$c(\mathcal{B}(\underline{g}))=2$.  
\end{corollary} 

\begin{remark}\label{remark:G_0-g}
Being a finitely generated reduced Krull monoid, $\mathcal{B}(\underline{g})\cong 
\mathcal{B}(H_0)$ for some finite subset $H_0$ in an abelian group $H$ (different from $G$ in general) by \cite[Theorem 2.7.14]{geroldinger_halter-koch} also when $g_1,\dots,g_m$ are not all distinct. 
However, the representation of our  monoid in the form $\mathcal{B}(\underline{g})$ is fundamental for our discussions. 

An easy direct proof of the isomorphism $\mathcal{B}(\underline{g})\cong 
\mathcal{B}(H_0)$ can be derived from the following observation. 
Take $\underline{g}\in G^m$, and suppose that $g_{m-1}=g_m$. 
Consider the group $H:=G\times \mathbb{Z}$, and the sequence 
\[\underline{h}:=((g_1,0),\dots,(g_{m-1},0),(g_{m-1},1), (0,-1))\in H^{m+1}.\]  
It is easy to see that we have a monoid isomorphism 
$\mathcal{B}(\underline{g})\cong \mathcal{B}(\underline{h})$. 
Now observe that there are less component repetitions in the sequence $\underline{h}$ then the number of component repetitions in $\underline{g}$. Note that the "price" for this manipulation was that we had to extend the group $G$. 
\end{remark}


\section{Gr\"obner bases} \label{sec:groebner}

In this section we give a Gr\"obner basis variant of the results of  
Section~\ref{sec:repetition}. 
Fix an admissible total order $\prec$ on the finitely generated free 
multiplicative monoid $M$; that is, $\prec$ is a total order such that for $x,y,z\in M$ with $x\prec y$ we have $xz\prec yz$,  and $1\prec x$ for each $x\in M\setminus \{1\}$. The latter condition ensures that 
the order $\prec$ is artinian, so any non-empty subset of $M$ contains a unique minimal element. Note also that $\prec$ is a term order of the polynomial ring $\mathbb{F}[M]$ (where $\mathbb{F}$ is a field) in the sense of Gr\"obner basis theory. 

\begin{definition}\label{def:groebner} A finite set 
$\Lambda\subset M\times M$ is a \emph{Gr\"obner system of the semigroup congruence $\sim$ on $M$} if the following conditions hold: 
\begin{itemize} 
\item[(i)] $x\sim y$ and $y\prec x$ for each $(x,y)\in \Lambda$; 
\item[(ii)]  $z\in M$ is the minimal element in its congruence class with respect to $\sim$ if 
there is no $(x,y)\in\Lambda$  such that $x$ divides $z$ in $M$.  
\end{itemize} 
\end{definition}  

\begin{proposition}\label{prop:groebner-basic} 
\begin{itemize} 
\item[(i)] If $\Lambda\subset M\times M$ is a Gr\"obner system of the semigroup congruence $\sim$ then $\Lambda$ generates $\sim$. 
\item[(ii)]  Every semigroup congruence $\sim$ on $M$ has a  Gr\"obner system. 
\item[(iii)] $\Lambda\subset M\times M$ is a Gr\"obner system of $\sim$ if and only if 
$\{x-y \mid (x,y)\in \Lambda\}$ is a Gr\"obner basis (satisfying $y\prec x$ for each of its elements $x-y$) of the ideal 
$\ker(\pi_{\mathbb{F}})$, 
where $\mathbb{F}$ is a field and $\pi_{\mathbb{F}}:\mathbb{F}[M]\to \mathbb{F}[M/\sim]$ is induced by the natural surjection $M\to M/\sim$. 
\end{itemize} 
\end{proposition}

\begin{proof} (i) Denote by $\sim_{\Lambda}$ the congruence generated by $\Lambda$. 
By assumption it is contained in $\sim$, since for each $(x,y)\in\Lambda$ we have $x\sim y$. To see the reverse inclusion it is sufficient to show that for any $z\in M$ we have  
$z\sim_{\Lambda}u$, where $u$ is the minimal element in the $\sim$-congruence class of $z$. 
If $z=u$, we are done. Otherwise $u\prec z$, hence by assumption there exists a pair 
$(x,y)\in \Lambda$ and $v\in M$ such that $z=xv$. Set $z_1=yv$. Then $z_1=yv\prec xv=z$ and 
$z\sim_{\Lambda} z_1$. 
If $z_1=u$, then we are done. Otherwise repeat the same step for $z_1$ instead of $z$   
(note that $z_1\sim z \sim u$). 
We obtain $z_2\in M$ with $z_2\prec z_1$ and $z_2\sim_{\Lambda}z_1$. If $z_2=u$ we are done, otherwise repeat the above step with $z_2$ instead of $z_1$. 
Since the order $\prec$ is artinian, in finitely many steps we must end up with 
$z\sim_{\Lambda}z_k=u$. 

(ii) It is well known that any binomial ideal has a  Gr\"obner basis consisting of binomials, see for example 
\cite[Lemma 8.2.17]{villareal}. Therefore the statement follows from (iii). 

(iii) Suppose $\{x-y \mid (x,y)\in \Lambda\}$ is a Gr\"obner basis of the ideal $\ker(\pi_{\mathbb{F}})$ (where $y\prec x$ for each $(x,y)\in \Lambda$). It follows that the initial ideal of $\ker(\pi_{\mathbb{F}})$ is generated by $L:=\{x\mid \exists y: (x,y)\in\Lambda\}$. Now take any $z\in M$ which is not minimal in its congruence class with respect to $\sim$. Then there is an $u\prec z$ such that $z\sim u$, 
so $z-u\in \ker(\pi_{\mathbb{F}})$ has initial term $z$. Therefore  there is an $x\in L$ such that $x$ divides $z$, so condition (ii) of Definition~\ref{def:groebner} holds for $\Lambda$ 
(it is obvious that condition (i) of Definition~\ref{def:groebner} holds for $\Lambda$). 

Conversely, assume that $\Lambda$ is a Gr\"obner system of $\sim$, and consider the subset $L:=\{x-y\mid (x,y)\in\Lambda\}$ in $\mathbb{F}[M]$.  Denote by $J$ the ideal generated by the initial terms of the elements in $L$. 
Clearly $L\subseteq \ker(\pi_{\mathbb{F}})$, hence $J$ is contained in the ideal $K$ generated by the initial terms of the ideal $\ker(\pi_{\mathbb{F}})$. 
By assumption the elements of $M\setminus J$ are all minimal in their congruence class with respect to $\sim$. In particular, they are pairwise incongruent, hence they are mapped by $\pi_{\mathbb{F}}$ to elements in $\mathbb{F}[M/\sim]$ that are linearly independent over $\mathbb{F}$. 
On the other hand, $M\setminus J\supseteq M\setminus K$, and the latter is mapped 
by $\pi_{\mathbb{F}}$ to an $\mathbb{F}$-vector space basis of  $\mathbb{F}[M/\sim]$ 
(see for example \cite[Proposition 1.1]{sturmfels}). 
It follows that $M\setminus J= M\setminus K$, implying in turn that $J=K$. 
The latter equality means that $L$ is a Gr\"obner basis of $\ker(\pi_{\mathbb{F}})$. 
\end{proof} 

We keep the notation of Section~\ref{sec:repetition}. 
In particular, $S$ is a reduced, affine, graded monoid, and $\tilde S$ is the monoid defined as in \eqref{eq:tildeS}. Fix an admissible total order $\prec$ on 
$M=\{x^{\alpha}\mid \alpha\in \mathbb{N}_0^{\mathcal{A}(S)}$. Define an admissible total order (denoted also by $\prec$) on the free monoid $\tilde M$ generated by 
$\{x_a\mid a\in\mathcal{A}(\tilde S)\}$ as follows. 
Enumerate the atoms in $\mathcal{A}(S)=\{a_1,\dots,a_n\}$ such that 
$x_{a_1}\prec x_{a_2} \prec\cdots \prec x_{a_n}$. Set $d_i=|a_i|$. 
For $\lambda,\mu\in \mathbb{N}_0^{\mathcal{A}(\tilde S)}$ we set $x^{\mu}\prec x^{\lambda}\in \tilde M$ if 
\begin{enumerate}
\item $x^{\kappa(\mu)}\prec x^{\kappa(\lambda)}$ in $(M,\prec)$; or 
\item  $\kappa(\mu)=\kappa(\lambda)$ and the sequence 
\begin{align*}\left(\mu(a_1[0]),\mu(a_1[1]),\dots,\mu(a_1[d_1]),\mu(a_2[0]),\dots,\mu(a_n[0]),\dots,
\mu(a_n[d_n])\right)\end{align*} 
is lexicographically greater than 
\[(\lambda(a_1[0]),\lambda(a_1[1]),\dots,\lambda(a_1[d_1]),\lambda(a_2[0]),\dots,\lambda(a_n[0]),\dots,
\lambda(a_n[d_n])).\]
\end{enumerate} 

\begin{theorem}\label{thm:groebner} 
Suppose that $\{(x^{\lambda},x^{\mu})\mid (\lambda,\mu)\in\Lambda\}$ is a Gr\"obner system of the semigroup congruence $\sim_S$. Then 
$\Gamma_1\cup\Gamma_2\cup \Gamma_3$ is a Gr\"obner system of the defining congruence $\sim_{\tilde S}$ of $\tilde S$, where 
\begin{align*}
\Gamma_1=\{(x^{\lambda},x^{\mu})\mid   (\kappa(\lambda),\kappa(\mu))\in \Lambda, 
\quad \delta(\lambda)=\delta(\mu)\} 
\end{align*}  
\begin{align*}
\Gamma_2=\{(x_{a[i]}x_{b[j]},x_{a[i-1]}x_{b[j+1]}) \mid a,b\in\mathcal{A}(S),\ x_a\prec x_b, \ 
0< i \le |a|,
\  0\le j < |b|\} 
\end{align*} 
\begin{align*} 
\Gamma_3=\{(x_{a[i]}x_{a[j]},x_{a[i-1]}x_{a[j+1]})\mid a \in\mathcal{A}(S),\ 0< i\le j< |a|\}.
\end{align*}
\end{theorem} 

\begin{proof} 
Take $(x^{\lambda},x^{\mu})\in \Gamma_1$. Then $(\kappa(\lambda),\kappa(\mu))\in \Lambda$, hence $x^{\kappa(\lambda)}\sim_S x^{\kappa(\mu)}$ and $x^{\kappa(\mu)}\prec x^{\kappa(\lambda)}\in M$. It follows that $x^{\mu}\prec x^{\lambda}\in \tilde M$. Moreover,   
$x^{\kappa(\lambda)}\sim_S x^{\kappa(\mu)}$ and $\delta(\lambda)=\delta(\mu)$ imply 
$x^{\lambda}\sim_{\tilde S} x^{\mu}$ by \eqref{eq:kappa_delta}. 
Therefore condition (i) of Definition~\ref{def:groebner} holds for the elements of $\Gamma_1$. It obviously holds for the elements of $\Gamma_2$ and $\Gamma_3$ by definition of the ordering $\prec$ on $\tilde M$. 

It remains to check that condition (ii) of Definition~\ref{def:groebner} holds for $\Gamma_1\cup \Gamma_2\cup \Gamma_3$. In order to do so, take $\lambda\in\mathbb{N}_0^{\mathcal{A}(\tilde S)}$ such that $x^{\lambda}\in \tilde M$ is not minimal in its congruence class with respect to $\sim_{\tilde S}$.  

Assume first that 
$x^{\kappa(\lambda)}\in M$ is not minimal in its congruence class with respect to $\sim_S$. 
Then by the assumption of the theorem on $\Lambda$, there exist $(\alpha,\beta)\in \Lambda $ and $\gamma\in \mathbb{N}_0^{\mathcal{A}(S)}$ such that 
$\kappa(\lambda)=\alpha+\gamma$. Clearly there exist 
$\tilde\alpha,\tilde\gamma\in \mathbb{N}_0^{\mathcal{A}(\tilde S)}$ with $\lambda=\tilde\alpha+\tilde\gamma$, $\kappa(\tilde\alpha)=\alpha$, and $\kappa(\tilde \gamma)=\gamma$. 
Also $x^{\alpha}\sim_S x^{\beta}$ implies 
$\sum_{a\in\mathcal{A}(S)}\alpha(a)|a|=\sum_{a\in\mathcal{A}(S)}\beta(a)|a|$. 
It is easy to infer from this equality the existence of $\tilde\beta\in \mathbb{N}_0^{\mathcal{A}(\tilde S)}$ with $\kappa(\tilde\beta)=\beta$ and $\delta(\tilde\beta)=\delta(\tilde\alpha)$. 
Moreover, $(\alpha,\beta)\in \Lambda$ implies $x^{\beta}\prec x^{\alpha}$, hence 
$x^{\tilde\beta}\prec x^{\tilde\alpha}$. So $(x^{\tilde\alpha}, x^{\tilde\beta})\in \Gamma_1$ by definition of $\Gamma_1$. Therefore $\Gamma_1$ testifies the non-minimality of $x^{\lambda}$ as it is required by (ii) of Definition~\ref{def:groebner}. 

Suppose next that $x^{\kappa(\lambda)}$ is minimal in its congruence class in $M$ with respect to 
$\sim_s$, and $x^{\lambda}\in \tilde M$ is not minimal in its congruence class with respect to $\sim_{\tilde S}$. It is easy to deduce from condition 2. of the definition of the ordering 
$\prec$ on $\tilde M$ that there must exist 
$(y,z)\in \Gamma_2\cup\Gamma_3$ such that $y$ divides $x^{\lambda}$. 
Consequently, the non-minimality of $x^{\lambda}$ is testified by $\Gamma_2\cup \Gamma_3$ as it is required by (ii) of Definition~\ref{def:groebner}. 
\end{proof}

\begin{remark}\label{remark:groebner}
The papers \cite{ohsugi_hibi} and \cite{shibuta} mentioned in Remark~\ref{remark:groebner} 
also give Gr\"obner bases of the binomial ideals considered there. 
\end{remark}

We call a Gr\"obner system $\Lambda$ {\it quadratic} if $|\lambda|\le 2$, $|\mu|\le 2$ for all 
$(\lambda,\mu)\in\Lambda$. 

\begin{corollary} \label{cor:quadratic_groebner}
If the semigroup congruence $\sim_S$ has a quadratic Gr\"obner system, then the 
the semigroup congruence $\sim_{\tilde S}$ also has a quadratic Gr\"obner system. 
\end{corollary} 

Note that if $S$ has a quadratic G\"obner system, then the semigroup algebra 
$\mathbb{C}[S]$ is Koszul (see \cite{polishchuk_positelski} for background on Koszul algebras). 
An iterated use of Corollary~\ref{cor:quadratic_groebner} yields the following: 

\begin{corollary}\label{cor:quadratic_B(g)} 
If $\mathcal{B}(\mathrm{supp}(\underline{g}))$ has a quadratic Gr\"obner system, then 
$\mathcal{B}(\underline{g})$ also has a quadratic Gr\"obner system, and hence 
the semigroup algebra $\mathbb{C}[\mathcal{B}(\underline{g})]$ is Koszul. 
\end{corollary}


\section{Relation to invariant theory}\label{sec:invariant} 

We need to recall a result from invariant theory. 
Let $H$ be a linearly reductive subgroup of the group $GL(V)$ of invertible linear transformations of a finite dimensional vector space $V$ over an algebraically closed field $\mathbb{F}$.  
The action of $H$ on $V$ induces an action via graded $\mathbb{F}$-algebra automorphims on the symmetric tensor algebra $S(V)$ of $V$ (graded in the standard way, namely $V\subset S(V)$ is the degree $1$ homogeneus component). 
Since $H$ is linearly reductive, the algebra $S(V)^H=\{f\in S(V)\mid h\cdot f=f \ \forall h\in H\}$ of polynomial invariants is known to be finitely generated. Let $f_1,\dots,f_n$ be a minimal homogeneous generating system of $S(V)^H$,  enumerated  so that 
$\deg(f_1)\ge\deg(f_2)\ge \cdots \ge \deg(f_n)$. Consider the 
$\mathbb{F}$-algebra surjection 
\begin{equation}\label{eq:S(V)^G_presentation}
\varphi:\mathbb{F}[x_1,\dots,x_n]\to S(V)^H \mbox{ with }x_i\mapsto f_i\quad  
(i=1,\dots,n). 
\end{equation} 
Endow $\mathbb{F}[x_1,\dots,x_n]$ with the grading diven by $\deg(x_i)=\deg(f_i)$, so 
$\varphi$ is a homomorphism of graded algebras. 
Recall that the factor of $S(V)$ modulo the ideal generated by $f_1,\dots,f_n$ is called the {\it algebra of coinvariants}. It is a finite dimensional graded vector space when $H$ is finite; in this case  
write $b(H,V)$ for its top degree (equivalently, all homogeneous elements in $S(V)$ of degree greater than $b(H,V)$ belong to the Hilbert ideal $S(V)f_1+\cdots+S(V)f_n$, and there is a homogeneous element in $S(V)$ of degree $b(H,V)$ not contained in the Hilbert ideal). 
Denote by $s$ the  Krull dimension of $S(V)^H$.  Note that $s\le n$ with equality only if 
$\ker(\varphi)=\{0\}$. 

\begin{theorem}\label{thm:derksen} {\rm (Derksen \cite[Theorems 1 and 2]{derksen})}  
\begin{itemize}
\item[(i)] We have the inequality 
$\displaystyle \mu(\ker(\varphi))\le \sum_{i=1}^{\min\{n,s+1\}}\deg(f_i)-s$, 
\item[(ii)]
When $H$ is finite,  we have the inequality 
$\mu(\ker(\varphi))\le 2b(G,V)+2$.  
\end{itemize} 
\end{theorem} 

The {\it Davenport constant of a finite subset $G_0$} of an abelian group $G$ is defined as 
\[\mathsf{D}(G_0)=\max\{|\alpha|\colon \alpha \in\mathcal{A}(\mathcal{B}(G_0))\},\] 
where $|\alpha|=\sum_{g\in G_0}\alpha(g)$. 
When  $G$ is finite, the {\it little Davenport constant of $G_0$} can be defined as 
\[\mathsf{d}(G_0)=\max\{|\alpha|\colon \forall \gamma\in\mathcal{A}(\mathcal{B}(G_0))
\ \exists g\in G_0 \mbox{ with }\gamma(g)>\alpha(g)\},\] 
the maximal length of a sequence over $G_0$ containing no product-one subsequence 
(see \cite[Proposition 5.1.3.2]{geroldinger_halter-koch}). 

Now let $\underline{g}=(g_1,\dots,g_m)$ be a sequence of elements from an arbitrary abelian group $G$, and use the notation developed in Section~\ref{sec:repetition}.  
Consider the following grading of the block monoid $\mathcal{B}(\underline{g})$: for $\alpha\in \mathcal{B}(\underline{g})$ its degree is $|\alpha|=\sum_{i=1}^m\alpha_i$. The graded catenary degree $c_{\mathrm{gr}}(\mathcal{B}(\underline{g}))$ is defined in Definition~\ref{def:graded_cat_degree} accordingly. 
Denote by $r(\mathcal{B}(\underline{g}))$ the rank of the free abelian subgroup 
in $\mathbb{Z}^m$ generated by $\mathcal{B}(\underline{g})$. 
Obviously $|\mathcal{A}(\mathcal{B}(\underline{g}))|\ge r(\mathcal{B}(\underline{g}))$ with equality if and only if 
$\mathcal{B}(\underline{g})$ is a free monoid. 
Set 
\[\mathcal{A}(\mathcal{B}(\mathrm{supp}(\underline{g}))):=\{a_1,\dots,a_n\}
\mbox{ with }|a_1|\ge |a_2|\ge \cdots \ge |a_n|.\] 

\begin{theorem}\label{thm:c_w(B(g))}
\begin{itemize} 
\item[(i)] We have the inequalities 
\[c_{\mathrm{gr}}(\mathcal{B}(\underline{g}))\le\max\{2|a_1|,c_{\mathrm{gr}}(\mathcal{B}(\mathrm{supp}(\underline{g})))\},\] 
and  
\[c_{\mathrm{gr}}(\mathcal{B}(\mathrm{supp}(\underline{g})))\le 
\sum_{i=1}^{\min\{n,r+1\}}|a_i| -r\]
where  $r=r(\mathcal{B}(\mathrm{supp}(\underline{g})))$. 
  
\item[(ii)] If $g_1,\dots,g_m$ generate a finite subgroup of $G$, then
\[c_{\mathrm{gr}}(\mathcal{B}(\underline{g}))\le 2 \mathsf{d}(\mathrm{supp}(\underline{g}))+2.\] 
\end{itemize}
\end{theorem} 

\begin{proof} 
We may assume that the components of $\underline{g}$ generate $G$. So $G$ is a finitely generated abelian group, whence it is isomorphic to $G_1\times \mathbb{Z}^k$, where 
$G_1$ is a finite abelian group, and $\mathbb{Z}^k$ is the free abelian group of rank $k$. 
Consider the linear algebraic group group $H=G_1\times T$, where $T$ is the torus 
$(\mathbb{C}^{\times})^k$. For an abelian linear algebraic group $A$ denote by 
$X(A)$ the group of homomorphisms $A\to \mathbb{C}^\times$ (as algebraic groups). 
Then $X(G_1)\cong G_1$ and $X(T)\cong \mathbb{Z}^k$, whence 
$X(H)\cong G_1\times \mathbb{Z}^k\cong G$. 
From now on we identify $G$ with $X(H)$. 
Let $V$ be a $\mathbb{C}$-vector space with basis $x_1,\dots,x_m$, and define a  action 
of $H$  on $V$ via linear transformations by setting 
$h\cdot x_i=g_i(h)x_i$ for $i=1,\dots,m$.  
The algebra $S(V)$ is the polynomial algebra 
$\mathbb{C}[x_1,\dots,x_m]$. The monomials span $1$-dimensional invariant subspaces,  and for $\alpha=(\alpha_1,\dots,\alpha_m)\in \mathbb{N}_0^m$ and $h\in H$ we have that 
\[h\cdot x^{\alpha}=\left(\prod_{i=1}^mg_i(h)^{\alpha_i}\right)x^{\alpha}. \]
It follows that the map 
$\mathbb{N}_0^m\to S(V)$, $\alpha\mapsto x^{\alpha}$ induces an isomorphism of the 
semigroup algebras 
\begin{equation}\label{eq:F[B(chi)]}
\mathbb{C}[\mathcal{B}(\underline{g})]\stackrel{\cong}\longrightarrow \mathbb{C}[x_1,\dots,x_m]^H.
\end{equation} 
The isomorphism \eqref{eq:F[B(chi)]} is an isomorphism of graded algebras. 
We may select as homogeneous generators of $S(V)^H$ the monomials 
$\{x^{\alpha}\mid \alpha\in \mathcal{A}(\mathcal{B}(\underline{g}))\}$. Then the presentation 
\eqref{eq:S(V)^G_presentation} of $S(V)^H$ is identified via \eqref{eq:F[B(chi)]} 
with the presentation 
\eqref{eq:pi_R} of the semigroup algebra $\mathbb{C}[\mathcal{B}(\underline{g})]$. So $\mu(\ker(\varphi))=\mu(\ker(\pi_{\mathbb{C}}))$ (see \eqref{eq:pi_R} in Section~\ref{sec:ringtheoretic-catenarydegree} 
for the definition of $\pi_{\mathbb{C}}:\mathbb{C}[M]\to \mathbb{C}[\mathcal{B}(\underline{g})]$). 
By Corollary~\ref{cor:mu=c} we know that 
$\mu(\ker(\pi_{\mathbb{C}}))=c_{\mathrm{gr}}(\mathcal{B}(\underline{g}))$. 
On the other hand apply Theorem~\ref{thm:derksen} for $\mu(\ker(\varphi))$ in the special case when $g_1,\dots,g_m$ are distinct (i.e. when 
$\mathcal{B}(\underline{g})=\mathcal{B}(\mathrm{supp}(\underline{g}))$), and 
combine it with Theorem~\ref{thm:tildeS_presentation} and Proposition~\ref{prop:B(tilde-g)}. 
Taking into account that the explanation of \eqref{eq:F[B(chi)]} shows also 
$b(G,V)=\mathsf{d}(\mathrm{supp}(\underline{g}))$ and that the Krull dimension of 
$\mathbb{C}[\mathcal{B}(\underline{g})]$ coincides with the rank of the free abelian subgroup of $\mathbb{Z}^m$ generated by $\mathcal{B}(\underline{g})$, we get the desired statements. 
\end{proof}

\begin{remark} When $G\cong \mathbb{Z}^k$, the group $H$ in the above proof is an algebraic torus, and the results in \cite{wehlau} give various bounds for  
$|a_1|$ in Theorem~\ref{thm:c_w(B(g))} (i). 
Moreover, \cite{wehlau:1994} characterizes the cases when $\mathbb{C}[\mathcal{B}(\underline{g})]\cong S(V)^H$ (for a torus $H$) is a polynomial ring, i.e. when 
$c_{\mathrm{gr}}(\mathcal{B}(\underline{g})=0$. 
\end{remark} 

\begin{corollary}\label{cor:bound-c_w(G_0)} 
For any subset $G_0$ of a finite abelian group $G$ we have the inequalities 
\[c_{\mathrm{gr}}(\mathcal{B}(G_0))\le 2\mathsf{d}(G_0)+2\le 2\mathsf{D}(G)\le 2|G|.\]  
\end{corollary}
\begin{proof} 
The first inequality is a special case of Theorem~\ref{thm:c_w(B(g))} (ii). 
To see the second inequality note the trivial inequality $\mathsf{d}(G_0)\le \mathsf{d}(G)$, and the well known equality $\mathsf{d}(G)+1=\mathsf{D}(G)$ 
(cf. \cite[Proposition 5.1.3.2]{geroldinger_halter-koch}). 
\end{proof} 

It follows immediately from Definitions~\ref{def:cat_degree} and 
\ref{def:graded_cat_degree} that 
\begin{equation}\label{eq:graded_cat_deg} 
c(S)\le \frac{1}{\min \{|a|\colon a\in \mathcal{A}(S)\}}c_{\mathrm{gr}}(S). 
\end{equation}

Therefore Theorem~\ref{thm:c_w(B(g))} implies bounds on the ordinary (not graded) catenary degree. For example, an immediate consequence of Corollary~\ref{cor:bound-c_w(G_0)} and \eqref{eq:graded_cat_deg} is the following:  
 
\begin{corollary} \label{cor:bound:c(G_0)} Let $G_0$ be a subset in a finite abelian group $G$. Then 
\[c(\mathcal{B}(G_0))\le 
\frac{2\mathsf{d}(G_0)+2}{\min \{|\alpha|\colon \alpha\in \mathcal{A}(\mathcal{B}(G_0))\}}\]
\end{corollary} 

As an application we recover the following known bound on $c(\mathcal{B}(G))$: 

\begin{corollary}\label{cor:c(G)<=D(G)} {\rm 
\cite[Theorem 3.4.10.5]{geroldinger_halter-koch}} 
For any finite abelian group $G$ 
we have the inequality $c(\mathcal{B}(G))\le \mathsf{D}(G)$. 
\end{corollary} 

\begin{proof} 
The monoid isomorphism 
$\mathcal{B}(G)\cong \mathcal{B}(G\setminus \{1_G\})\times \mathcal{B}(\{1_G\})
\cong  \mathcal{B}(G\setminus \{1_G\})\times \mathbb{N}_0$ 
implies that 
$c(\mathcal{B}(G))=c(\mathcal{B}(G\setminus \{1_G\})$. 
For a nontrivial group $G$ the minimal degree of an atom in 
$\mathcal{B}(G\setminus \{1_G\})$ is $2$, hence Corollary~\ref{cor:bound:c(G_0)}
gives $c(\mathcal{B}(G\setminus \{1_G\})\le \frac{2\mathsf{d}(G)+2}{2}=\mathsf{d}(G)+1
=\mathsf{D}(G)$. 
\end{proof}


\section{Relation to toric varieties}\label{sec:toric}

The quotient construction of toric varieties (cf. \cite{cox}) represents a toric variety as the categorical quotient of a Zariski open subset in a vector space endowed with an action of a diagonalizable group (see \cite{cox_little_schenck} for background on toric varieties). 
Rings of invariants are at the basis of quotient constructions in algebraic geometry. In the proof of Theorem~\ref{thm:c_w(B(g))} we recalled that the ring of invariants $\mathbb{C}[x_1,\dots,x_m]^H$ of a diagonalizable group action is isomorphic to a semigroup ring $\mathbb{C}[\mathcal{B}(\underline{g})]$ of a block monoid. 
Therefore the results in Sections~\ref{sec:repetition} \ref{sec:groebner}, 
\ref{sec:invariant} have relevance for toric varieties. 

In more details, the coordinate rings  of affine toric varieties with no torus factors are 
the semigroup rings (over $\mathbb{C}$) of reduced, affine Krull monoids. 
This class of rings (up to isomorphism) is the same as the class of rings of invariants  
$\mathbb{C}[x_1,\dots,x_m]^H$, where $H$ is an abelian group, and each variable spans an  
$H$-invariant subspace (see for example \cite[Corollary 5.19]{bruns_gubeladze}), 
which is the same as the class of rings of the form $\mathbb{C}[\mathcal{B}(\underline{g})]$.  

Projective toric varieties can be constructed as the projective spectrum of semigroup algebras of reduced affine Krull monoids, see for example \cite[Chapter 10]{miller_sturmfels},  \cite[Theorem 14.2.13]{cox_little_schenck}. Namely, take 
$\underline{g}=(g_1,\dots,g_m)\in G^m$ such that $\mathcal{B}(\underline{g})=\{0\}$, 
and fix an element $h\in G$. 
Endow the monoid $\mathcal{B}((\underline{g},h))$ with the grading given by 
$\mathcal{B}((\underline{g},h))_d=\{\alpha\in \mathcal{B}((\underline{g},h))\subseteq \mathbb{N}_0^{m+1}\mid \alpha_{m+1}=d\}$, $d=0,1,2,\dots$. 
Then $\mathbb{C}[\mathcal{B}((\underline{g},h))]$ becomes a graded algebra, whose projective spectrum is a projective toric variety. 

\begin{example}\label{example:toric_quiver} We reformulate a result on presentations of homogeneous coordinate rings of projective toric quiver varieties from \cite{domokos_joo} in the terminology of the present paper. Let $\Gamma$ be an acyclic quiver (i.e. a finite directed graph having no oriented cycles), with vertex set $\{1,\dots,k\}$ and arrow set 
$\{e_1,\dots,e_m\}$. For an arrow $e_i$ denote by $s(e_i)$ the starting vertex of $e_i$, and denote by $t(e_i)$ the terminating vertex of $e_i$. In the additive group $\mathbb{Z}^k$ 
consider the elements $g_i=(g_{i1},\dots,g_{ik})$, $i=1,\dots,m$ given by 
\[g_{ij}=\begin{cases} -1, &\mbox{ if }j=s(e_{i}) \\
1 &\mbox{ if } j=t(e_i) \\
0&\mbox{ otherwise.}
\end{cases}\]
Pick an element $h\in \mathbb{Z}^k$ whose additive inverse is contained in the subgroup of $\mathbb{Z}^k$ generated by $g_1,\dots,g_m$. Then \cite[Theorem 9.3]{domokos_joo} asserts that the catenary degree $c(\mathcal{B}((\underline{g},h)))$ of $\mathcal{B}((\underline{g},h))$ is at most $3$. 
Moreover, it is shown in \cite{domokos_joo:2} that if we assume in addition that if 
$r(\mathcal{B}((\underline{g},h)))\le 5$ (i.e. the corresponding toric variety has dimension at most $4$), then $c(\mathcal{B}((\underline{g},h)))\le 2$ with essentially one exception. 
We mention that presentations of the coordinate ring of affine toric quiver varieties are 
studied in \cite{joo}.  
\end{example}

Well known open conjectures (of increasing strength) in combinatorial commutative algebra are the following (called sometimes B{\o}gvad's conjecture; see 
\cite[Conjecture 13.19]{sturmfels} or \cite{bruns}): Given  a smooth, projectively normal projective toric variety, its 
\begin{itemize} 
\item[(i)] vanishing ideal is generated by quadratic elements.  
\item[(ii)] homogeneous coordinate ring is Koszul. 
\item[(iii)] vanishing ideal has a quadratic G\"obner basis.  
\end{itemize}

For a finite subset $G_0$ of $G$ with $\mathcal{B}(G_0)=\{0\}$ and an element $h$ whose inverse belongs to the subgroup generated by $G$, the monoid 
$\mathcal{B}(G_0\cup\{h\})$ is endowed with the grading such that the degree $d$ component consists of the elements $\alpha$ in $\mathcal{B}(G_0\cup\{h\})$ with $\alpha(h)=d$. 
Suppose that $\mathcal{B}(G_0\cup\{h\})$ has a quadratic Gr\"obner system. According to the above conjecture this is expected to happen when $\mathcal{B}(G_0\cup\{h\})$ is generated in degree $1$ (so $\mathcal{B}(G_0\cup\{h\})$ is half-factorial in the sense of factorization theory), and the projective spectrum of 
$\mathbb{C}[\mathcal{B}(G_0\cup\{h\})]$ 
 is a smooth projective variety. (For instance, in the setup of 
 Example~\ref{example:toric_quiver} this holds by 
 \cite{haase} for almost all choices of $h$ when $\Gamma$ is a bipartite directed graph with $3$ source and $3$ sink vertices.) 
Then for any $\underline{g}$ with $\mathrm{supp}(\underline{g})=G_0$ we have 
by Corollary~\ref{cor:quadratic_B(g)} that $\mathcal{B}((\underline{g},h))$ has a quadratic 
Gr\"obner basis,  and hence the algebra $\mathbb{C}[\mathcal{B}((\underline{g},h))]$ is Koszul (although its projective spectrum typically fails to be a smooth projective variety). 


\begin{center} Acknowledgements. \end{center} 
I thank K\'alm\'an Cziszter, D\'aniel Jo\'o, and Szabolcs M\'esz\'aros for several discussions 
on the topic of this paper.  I am also grateful to Alfred Geroldinger for helpful comments on the manuscript.


\end{document}